\documentclass[12pt]{amsart} 
\usepackage{amsmath}
\usepackage{amssymb}
\usepackage{amsthm}
\usepackage{mathtools}

\usepackage[utf8]{inputenc}
\usepackage[english]{babel}
\usepackage{color}

\usepackage{wrapfig}

\newtheorem{theorem}{Theorem}[section]

\newtheorem{lemma}[theorem]{Lemma}

\newtheorem{proposition}[theorem]{Proposition}

\newtheorem{remark}{Remark}

\title{Falconer-type estimates for dot products}
\date{\today}
\author{Alex Iosevich}
\address{Department of Mathematics, University of Rochester, Rochester, NY 14627}
\email{iosevich@math.rochester.edu}
\thanks{This material is based on work supported by the National Science Foundation under grant no. HDR TRIPODS - 1934962}
\author{Steven Senger}
\address{Department of Mathematics, Missouri State University, Springfield, MO 65897}
\email{stevensenger@missouristate.edu}

\begin{document}
\maketitle
\begin{abstract} We present a family of sharpness examples for Falconer-type single dot product results. In particular, for $d\geq 2,$ for any $s<\frac{d+1}{2},$ we construct a Borel probability measure $\mu$ satisfying the energy estimate $I_s(\mu)<\infty,$ yet the estimate 
\begin{equation} \label{monalisa} (\mu \times \mu)\{(x,y):1\leq x\cdot y \leq 1+\epsilon\} \leq C\epsilon \end{equation} does not hold with constants independent of $\epsilon$. It is known (\cite{EIT11}) that such an estimate always holds with $C$ independent of $\epsilon$ if $I_{\frac{d+1}{2}}(\mu)<\infty$. Thus our estimate proves the sharpness of the dimensional threshold in this result and generalizes similar results (\cite{Mat95}, \cite{IS16}) established in the case when the dot product $x \cdot y$ is replaced by the Euclidean distance function $|x-y|$, or, more generally, ${||x-y||}_K$, the distance that comes from the norm induced by a symmetric convex body $K$ with a smooth boundary and non-vanishing curvature. Our constructions are partially based on ideas that come from discrete incidence theory.

\end{abstract}

\section{Introduction}
In \cite{Fal86}, Falconer conjectured that any subset of $\mathbb R^d$ with Hausdorff dimension greater than $\frac{d}{2}$ would determine a set of distances with positive Lebesgue measure. In the same paper, he proved that a weaker version of the conjecture holds for sets whose Hausdorff dimension is more than $\frac{d+1}{2}.$ While this threshold has been lowered over the years (see \cite{Mat87, Mat95, W99, Erd05, GIOW}), the full conjecture is still open. The key to Falconer's original result was based on the following lemma, which states that if a certain energy estimate holds, then we are guaranteed a bound on the product measure of pairs of points separated by approximately any fixed distance.
\begin{lemma}\label{key}
Given a compactly supported Borel measure $\mu$, there exists a constant $C>0$ such that
\begin{equation}\label{muIncidence} (\mu\times\mu)\{(x,y): 1 \leq |x-y| \leq 1+\epsilon\} \leq C I_{\frac{d+1}{2}}(\mu) \epsilon,\end{equation} where 
\begin{equation}\label{muEnergy} I_s(\mu) \equiv \int\int |x-y|^{-s}d\mu(x)d\mu(y). \end{equation} \end{lemma}

In particular, this lemma says that if $\mu$ is supported on a set of Hausdorff dimension $>\frac{d+1}{2}$, then the estimate (\ref{muIncidence}) holds with uniform constants. It is also not difficult to see that the proof easily extends to the case when the Euclidean norm $|\cdot|$ is replaced by ${|| \cdot ||}_K$, the norm induced by a symmetric convex body $K$ with a smooth boundary and non-vanishing Gaussian curvature. 

In \cite{Mat95}, Mattila showed that Falconer's lemma is sharp in the sense that if $d=2$, then for any $s<\frac{3}{2},$ there exists a measure obeying $I_s(\mu)<\infty$, and yet 
$$ \limsup_{\epsilon \to 0} \epsilon^{-1} (\mu\times\mu)\{(x,y): 1 \leq |x-y| \leq 1+\epsilon\}=\infty.$$ 

This result was extended to three dimensions by the authors of this paper (\cite{IS16}). In higher dimensions, the sharpness of Falconer's lemma is still open in the case of the Euclidean distance, but in the same paper, the authors of this showed that for any $s<\frac{d+1}{2},$ there exists a compactly supported Borel measure $\mu$ with $I_s(\mu)<\infty$, $s<\frac{d+1}{2}$,  and yet 
$$ \limsup_{\epsilon \to 0} \epsilon^{-1} (\mu\times\mu)\{(x,y): 1 \leq {||x-y||}_K \leq 1+\epsilon\}=\infty,$$ where ${||\cdot||}_K$ is the norm induced by a symmetric convex body $K$ obtained by gluing the paraboloid in such a way that results in a smooth symmetric convex body with a smooth boundary and non-vanishing Gaussian curvature. 

In recent decades, the Falconer distance problem has been generalized in a variety of directions, each having its own geometric, analytic, and combinatorial nuances. For example, given a compact set $E$ in ${\Bbb R}^d$, $d \ge 2$, we can ask whether the dot product set 
$$\Pi(E)=\{x \cdot y: x,y \in E\}$$ has positive Lebesgue measure. A step in this direction was taken by Suresh Eswarathasan, the first listed author of this paper, and Krystal Taylor (\cite{EIT11}) who proved that the Lebesgue measure of $\Pi(E)$ is indeed positive if the Hausdorff dimension of $E$ is greater than $\frac{d+1}{2}$. They proved this by showing that the analog of Lemma \ref{key} holds if $|x-y|$ is replaced by $x \cdot y$, i.e. 
\begin{equation} \label{dpfalconer} (\mu \times \mu)\{(x,y): 1 \leq x \cdot y \leq 1+\epsilon \} \leq C I_{\frac{d+1}{2}}(\mu) \epsilon. \end{equation} 

Indeed, they proved a much more general result where the conclusion of the Falconer lemma holds if $|x-y|$ is replaced by any $\phi(x,y)$ smooth away from the diagonal and having a non-zero Monge-Ampere determinant. In this paper we prove that (\ref{dpfalconer}) is sharp. To state this precisely, we introduce some notation and state our main result. Here and throughout, $X \lesssim Y$ means that there exists a uniform constant $C>0$ such that $X \leq CY$. Also, we write $X\approx Y$ to mean that both $X\lesssim Y$ and $Y\lesssim X.$

\begin{theorem}\label{main}
For any dimension $d\geq 3,$ and any $s\in\left(\frac{d}{2},\frac{d+1}{2}\right),$ there exists a Borel measure $\mu$ on $\mathbb R^d$ such that $I_s(\mu)\approx 1$ and for any $\epsilon >0,$
$$(\mu \times \mu) \left\{(x,y) : 1 \leq x \cdot y \leq 1+ \epsilon \right\} \approx \epsilon^{\frac{2s}{d+1}}.$$

In particular, 
$$ \limsup_{\epsilon \to 0} \epsilon^{-1} (\mu \times \mu) \left\{(x,y) : 1 \leq x \cdot y \leq 1+ \epsilon \right\}=\infty.$$ 
\end{theorem}

\vskip.125in 

\begin{remark} The analogous statement of Theorem \ref{main} for two dimensions was shown by Suresh Eswarathasan, the first listed author of this paper, and Krystal Taylor (\cite{EIT11}) using a simpler construction. The construction used in this paper can also work when $d=2,$ but the estimates become more delicate, and the analysis becomes much more cumbersome than for it is for $d\geq 3.$ Because a simpler proof for the two dimensional case already exists, we only prove the result for higher dimensions here.
\end{remark}

\begin{remark} The authors believe that the estimate 
\begin{equation} \label{generalMAest} (\mu \times \mu) \{(x,y): t \leq \phi(x,y) \leq t+\epsilon \} \lesssim I_{\frac{d+1}{2}}(\mu) \epsilon, \end{equation} which is shown in (\cite{EIT11}) to hold for all functions $\phi$ satisfying the non-vanishing Monge-Ampere determinant condition, is sharp in the sense that for any $s<\frac{d+1}{2}$ there exists a compactly supported Borel measure $\mu$ with $I_s(\mu)<\infty$ for which (\ref{generalMAest}) fails. We hope to address this issue in the sequel. \end{remark} 

The authors would like to thank Adam Sheffer for helpful conversations about the discrete incidence construction in this paper, and Thang Pham for pointing out an error in an earlier version.

\vskip.25in 

\section{Proof of Theorem \ref{main}} 

\vskip.125in 

\subsection{Preliminaries}
We begin with the celebrated Szemer\'edi-Trotter Theorem, from \cite{ST83}.
\begin{theorem}\label{ST}[Szemer\'edi-Trotter]
Given a set of $n$ points and $m$ lines in $\mathbb R^2$, the number of incidences is bounded above by
$$I \lesssim (mn)^\frac{2}{3}+m+n.$$
\end{theorem}
We now record some illustrative discrete constructions and indicate why they cannot be extended to the continuous setting as in Theorem \ref{main}.
\begin{proposition}\label{silly}
In $\mathbb R^2,$ for any large finite $n\in \mathbb N$, there exists a set $E$ of $n$ points with $\approx n^2$ occurrences of the dot product zero. In $\mathbb R^3,$ for any large finite $n\in \mathbb N$, there exists a set $E$ of $n$ points with $\approx n^2$ occurrences of any dot product.
\end{proposition}
\begin{proof} In $\mathbb R^2,$ for we arrange $n/2$ points along the $x$-axis and $n/2$ points along the $y$-axis. Note that any point on the $x$-axis is orthogonal to any point on the $y$-axis, and there are $n^2/4$ such point pairs. In $\mathbb R^3,$ for any $\alpha \in \mathbb R,$ we arrange $n/2$ points along the line $\{(1,y,0):y\in \mathbb R\}$ and arrange $n/2$ points along the line $\{(\alpha,0, z):z\in \mathbb R\}$. Similar to the previous case, the dot product of any point from the first line with any point from the second line will be $\alpha$, and we will again have $n^2/4$ such point pairs.
\end{proof}
In a rough sense, both of the constructions in Proposition \ref{silly} are too ``low-dimensional" to be used to construct a measure as in Theorem \ref{main}. To quantify this, if we follow the procedure detailed below for either of these constructions, we will not get the corresponding energy bound, which would lead to an unbounded $I_s(\mu)$ for $s>1$. It should be noted that both of these examples were inspired by the celebrated Lenz construction in $\mathbb R^4$. It consists of $n/2$ points on a unit circle in the first two dimensions and $n/2$ points on a circle in the second two dimensions, and has $n^2/4$ occurrences of the distance $\sqrt 2$, measured between points from the different circles.

\subsection{Constructing the measure}
First, for any large, finite $n\in \mathbb N$, we will construct a set of $\approx n$ points in $[0,2]^d$ that has $\approx n^\frac{2d}{d+1}$ pairs of points whose dot product is 1. This construction is motivated by a well-known sharpness example for Theorem \ref{ST}. Next, we will quantify some properties of the point set. Finally, we will use an infinite sequence of such sets to generate a measure $\mu$ satisfying the properties of Theorem \ref{main}.
\subsubsection{Discrete construction}\label{discreteSet}
Fix a large, finite, $q\in \mathbb N,$ such that $q^{d+1}\approx n$. Define
$$A:=\left\{\frac{q+i}{2q}:i=1,\dots, q\right\},$$
and
$$B:=\left\{\frac{q^2+i}{2q^2}:i=1,\dots,q^2\right\}.$$
Let $A^k$ denotes a Cartesian product of $k$ copies of $A$, and define $$E:=A^{d-1}\times B \subset [0,1]^2.$$ We can see that $|E|=q^{d+1}.$ Next, we turn our attention to another set of points defined by $A$ and $B$. Let $(-A)$ denote the additive complements of $A$, namely $(-A):=\{-a:a\in A\}.$
$$F:=\left\{\left(\frac{-m_1}{b}, \frac{-m_2}{b}, \dots, \frac{-m_{d-1}}{b}, \frac{1}{b}\right):m_{j}\in(-A), b\in B \right\}.$$
Finally, given a $d$-tuple, $(m_1, m_2, \dots, m_{d-1},b),$ define the hyperplane
$$H(m_1, m_2, \dots, m_{d-1},b) := \left\{ (x_1, x_2, \dots, x_d)\in \mathbb R^d : x_d = \left(\sum_{j=1}^{d-1}m_jx_j\right) + b \right\}.$$
The family of these hyperplanes we will consider is
$$\mathcal H :=\left\{H(m_1, m_2, \dots, m_{d-1},b):m_j\in(-A), b\in B \right\}.$$
We now show that these hyperplanes are level sets for points in $F.$
\begin{lemma}\label{hyp}
For every $f\in F$, there is a unique $H_f\in\mathcal H$ so that for any $x\in H_f,$ we have $f\cdot x =1.$
\end{lemma}
\begin{proof}
To see this, fix a $d$-tuple, $(m_1, m_2, \dots, m_{d-1},b)\in(-A)^{d-1}\times B,$ and compute the dot product of the element of $F$ and any point on the hyperplane associated to the same $d$-tuple. So the associated point $f\in F$ will be
$$f= \left(\frac{-m_1}{b}, \frac{-m_2}{b}, \dots, \frac{-m_{d-1}}{b}, \frac{1}{b}\right),$$
and we will consider an arbitrary point $x\in H(m_1, m_2, \dots, m_{d-1},b).$ Their dot product will be
\begin{align*}
f\cdot x &= \left(\frac{-m_1}{b}, \frac{-m_2}{b}, \dots, \frac{-m_{d-1}}{b}, \frac{1}{b}\right) \cdot \left(x_1, x_2, \dots, x_{d-1}, \left(\sum_{j=1}^{d-1}m_jx_j \right)+b \right)\\
&=\left(\sum_{j=1}^{d-1}\frac{-m_jx_j}{b}\right) + \left(\sum_{j=1}^{d-1}\frac{m_jx_j}{b}\right) +\frac{b}{b}=1.
\end{align*}
\end{proof}
We now pull this all together to quantify how many point pairs in $E\cup F$ determine the dot product 1.
\begin{proposition}\label{discreteDP}
There are $\approx q^{2d}$ point pairs determining the dot product 1 in $E\cup F.$
\end{proposition}
\begin{proof}
Fix an arbitrary point $f\in F$. Notice that the associated hyperplane, $H_f\in \mathcal H,$ consists of points that have dot product 1 with $f.$ However, each $H_f$ will contain $\approx q^{d-1}$ points from $E$. Since there are $q^{d+1}$ choices for $f,$ we have a total of $\approx q^{d+1}q^{d-1}=q^{2d}$ point pairs whose dot product is 1, as claimed.
\end{proof}

\subsubsection{Separation}

In order to construct our measure from the discrete point sets $E$ and $F$, we will need to show that these sets are not too ``low-dimensional" as were the point sets from Proposition \ref{silly}. To be sure we will need to show that the points are separated, and that they will not concentrate mass too much. We make this precise below.

By the definitions of $A$ and $B$, we know that the minimum difference between distinct coordinates of points in $E$ is $q^{-2}.$ This gives us
\begin{equation}\label{Esep}
\min_{\substack{p, p'\in E\\ p\neq p'}}{|p-p'|}\gtrsim \frac{1}{q^2}.
\end{equation}
However, $F$ will take a little more work.

\begin{lemma}\label{Sep}
$$\min_{\substack{p, p'\in F\\ p\neq p'}}{|p-p'|}=\frac{1}{q^2}.$$
\end{lemma}
\begin{proof}
Consider two arbitrary distinct points, $p,p'\in F.$
$$p= \left(\frac{-m_1}{b}, \frac{-m_2}{b}, \dots, \frac{-m_{d-1}}{b}, \frac{1}{b}\right),$$
and
$$p'= \left(\frac{-m_1'}{b'}, \frac{-m_2'}{b'}, \dots, \frac{-m_{d-1}'}{b'}, \frac{1}{b'}\right),$$
where the $m_j$ and $m_j'$ come from $A$ and $b,b'\in B.$ We now split into two cases: the case where $b\neq b',$ and the case where $b=b'.$

If $b\neq b',$ we have that, for appropriate choices of $1 \leq i_d,i_d'\leq q^2,$
$$|p-p'|\geq \left|\frac{1}{b}-\frac{1}{b'}\right| = \left|\frac{b'-b}{bb'}\right|=\left|\frac{\frac{q^2+i_d'}{2q^2}-\frac{q^2+i_d}{2q^2}}{\left(\frac{q^2+i_d}{2q^2}\right)\left(\frac{q^2+i_d'}{2q^2}\right)}\right|$$
$$=2q^2\left| \frac{(q^2+i_d')-(q^2+i_d)}{(q^2+i_d)(q^2+i_d')}\right|=2q^2\left|\frac{i_d'-i_d}{q^4+q^2(i_d+i_d')+i_di_d'}\right|,$$
which, because $i_d,i_d'\leq q^2,$ is bounded below by
$$\geq 2q^2\left|\frac{i_d'-i_d}{q^4+q^2(q^2+q^2)+q^2q^2} \right|\geq 2q^2\left|\frac{1}{q^4+q^2(q^2+q^2)+q^2q^2} \right|\gtrsim q^{-2},$$
where we used the fact that $i_d\neq i_d'$ in the last line. Putting this together, we obtain that for $p$ and $p'$ whose last coordinates are different,
\begin{equation}\label{dthDiff}
|p-p'|\gtrsim q^{-2}.
\end{equation}

If $b=b',$ then for $p$ and $p'$ to be distinct points, they must differ in some other coordinate. Suppose that $p$ has the value $\frac{-m}{b}$ in that coordinate, and $p'$ has the value $\frac{-m'}{b}$ in the same coordinate, for distinct $m,m'\in (-A)$. Then we can be assured that for appropriate $1 \leq i,i'\leq q$ and $1\leq i_d\leq q^2,$ we have
$$|p-p'|\geq\left|\frac{-m}{b}-\frac{-m'}{b}\right|=\left|\frac{-\left(\frac{-(q+i)}{2q}\right)}{\frac{q^2+i_d}{2q^2}}-\frac{-\left(\frac{-(q+i')}{2q}\right)}{\frac{q^2+i_d}{2q^2}}\right|$$
$$=q\left|\frac{(q+i)-(q+i')}{q^2+i_d}\right|\geq q\left|\frac{i-i'}{q^2+i_d} \right|,$$
which, because $i_d \leq q^2,$ is bounded below by
$$\geq q\left| \frac{i-i'}{2q^2}\right|\geq q\frac{1}{2q^2}\gtrsim q^{-1},$$
where we used the fact that $i\neq i',$ as they correspond to the distinct values of $m$ and $m'$ in the coordinate where $p$ and $p'$ differ. Putting these together, we get for distinct $p$ and $p'$ with the same last coordinate,
\begin{equation}\label{ithDiff}
|p-p'|\gtrsim q^{-1}.
\end{equation}
Combining \eqref{dthDiff} and \eqref{ithDiff} yields the desired result.
\end{proof}
Notice that by \eqref{Esep} and Lemma \ref{Sep}, we get the following.
\begin{proposition}\label{sep}
There are no more than a constant number of points of $E \cup F$ in any cube of side-length $q^{-2}.$
\end{proposition}
\subsubsection{Discrete energy}
Now that we know our discrete points are sufficiently separated, we turn our attention to showing that they are also not clumped up too much. We quantify this by calculating the discrete versions of the associated energy integrals, namely
$$I_s'(X,Y)=\frac{1}{{n \choose 2}}\sum_{\substack{p \in X,\\ p'\in Y,\\ p\neq p'}}|p-p'|^{-s},$$
with the convention that we write $I_s'(X)$ in place of $I_s'(X,X).$ Also notice that $I_s'(X,Y)=I_s'(Y,X),$ by definition.
We will break the total discrete energy of our point set into three parts, each treated separately:
$$I_s'(E \cup F)= I_s'(E) + 2I_s'(E,F) + I_s'(F).$$
The discrete energy calculation for $E$ is given as Lemma 2.1 in \cite{IS16}. This gives
\begin{equation}\label{enE}
I_s'(E)\lesssim 1.
\end{equation}
Now we compute the discrete energy for $F$, recalling that $n\approx q^{d+1}.$
$$
I_s'(F)=\frac{1}{{n \choose 2}}\sum_{\substack{p,p' \in E,\\ p\neq p'}}|p-p'|^{-s}\approx q^{-2d-2}\sum_{\substack{p,p' \in F,\\ p\neq p'}}|p-p'|^{-s}= I + II,
$$
where $I$ is the sum over $p$ and $p'$ whose first $(d-1)$ coordinates are the same, and $II$ is its complement.
\subsubsection{Bounding $I$}
To estimate $I$, we mirror the derivation of \eqref{dthDiff}, and use similar notation for coordinates of points. Define $\mathcal F$ to be the set of point pairs of $F$ that agree in the first $(d-1)$ coordinates. That is,
$$\mathcal F :=\{(p, p')\in F\times F :p_d\neq p_d', p_j=p_j', j=1,\dots,d-1\}.$$
 Because the sum is over pairs of points that agree on the first $(d-1)$ coordinates, we only need to focus on the difference in the final coordinate. Recall that coordinates $p_j$ are written in terms of their respective indices $i_j$ as before.
\begin{align*}
I &\approx n^{-2}\sum_{(p, p')\in \mathcal F}|p_d-p_d'|^{-s}\lesssim q^{-2d-2}\sum_{p\in F}\sum_{\substack{i_d'\in[1..q^2],\\ i_d'\neq i_d}}\left|\frac{2q^2}{q^2+i_d}-\frac{2q^2}{q^2+i_d'}\right|^{-s}\\
&\leq q^{-2d-2}\sum_{i_1=1}^q\cdots\sum_{i_{d-1}=1}^q\sum_{i_d=1}^{q^2}\sum_{\substack{i_d'\in[1..q^2],\\ i_d'\neq i_d}}(2q^2)^{-s}\left|\frac{1}{q^2+i_d}-\frac{1}{q^2+i_d'}\right|^{-s}\\
&\lesssim q^{-2s-2d-2}q^{d-1}\sum_{i_d=1}^{q^2}\sum_{\substack{i_d'\in[1..q^2],\\ i_d'\neq i_d}}\left|\frac{1}{q^2+i_d}-\frac{1}{q^2+i_d'}\right|^{-s}.
\end{align*}
The partial sum with $i_d'<i_d$ is the same as its complement, so we rewrite it as twice the sum with $i_d'>i_d$. We will also use the fact that $1\leq i_d,i_d'\leq q^2$ to bound the denominator below.
\begin{align*}
I &\lesssim q^{-2s-d-3}\cdot 2 \sum_{i_d=1}^{q^2-1}\sum_{i_d' =i_d+1}^{q^2}\left|\frac{(q^2+i_d')-(q^2+i_d)}{q^4+q^2(i_d+i_d')+i_di_d'}\right|^{-s}\\
&\lesssim q^{-2s-d-3}\sum_{i_d=1}^{q^2-1}\sum_{i_d' = i_d+1}^{q^2}\left|\frac{i_d'-i_d}{4q^4}\right|^{-s}\lesssim q^{-2s-d-3} q^{4s}\sum_{i_d=1}^{q^2-1}\sum_{i_d' = i_d+1}^{q^2}(i_d'-i_d)^{-s}.
\end{align*}
We reparameterize the sum by setting $j:=i_d'-i_d$. We then approximate the sum by an integral.
\begin{align*}
I &\lesssim q^{2s-d-3}\sum_{i_d=1}^{q^2-1}\sum_{i_d' = i_d+1}^{q^2}(i_d'-i_d)^{-s}=q^{2s-d-3}\sum_{j=1}^{q^2-1}(q^2-j)j^{-s}\\
&\approx q^{2s-d-3}\int_1^{q^2}(q^2-x)x^{-s}dx\\
&=q^{2s-d-3}\left(\left[(q^2-x)-\left(\frac{x^{1-s}}{1-s}\right)\right]_1^{q^2}-\left[ \frac{x^{2-s}}{(1-s)(2-s)}\right]_1^{q^2} \right)\\
&\approx q^{2s-d-3}\left(q^2-\frac{q^{4-2s}}{(1-s)(2-s)} \right).
\end{align*}
Recall that by assumption, we have that $s>(d/2)>1,$ and notice that if $1<s<2,$ the second term in the parentheses is still bounded by a constant times $q^2.$ Putting this all together with the assumption that $s<(d+1)/2,$ we get
$$I\lesssim q^{2s-d-1}\lesssim 1.$$
\subsubsection{Bounding $II$}
The estimate of $II$ is similar, though significantly more involved.
\begin{align*}
II &= q^{-2d-2}\sum_{\substack{p\neq p' \in F,\\ (p, p')\notin \mathcal F}}|p-p'|^{-s}.
\end{align*}
We now estimate $II$ by replacing the $\ell^2$ distance by $\ell^1$ distance, losing at most a constant in the process. 
\begin{align*}
II&\leq q^{-2d-2}\sum_{\substack{p\neq p' \in F,\\ (p, p')\notin \mathcal F}}\left(\sum_{j=1}^d |p_j-p_j'|^2\right)^{-\frac{s}{2}}\lesssim q^{-2d-2}\sum_{\substack{p\neq p' \in F,\\ (p, p')\notin \mathcal F}}\left(\sum_{j=1}^d |p_j-p_j'|\right)^{-s}.
\end{align*}
We then rearrange the terms of the sum so that the largest difference between coordinate indices from the first $(d-1)$ is recorded in the first coordinate, with $i_1>i_1'.$ So we can continue bounding the above sum by
\begin{align*}
&\leq q^{-2d-2}2(d-1)\sum_{i_1'=1}^{q-1}\sum_{i_1=i_1'+1}^q\sum_{\substack{1\leq i_j,i_j'\leq q\\j=2, \dots, (d-1)}}\sum_{1\leq i_d,i_d' \leq q^2}\left(\sum_{j=1}^d |p_j-p_j'|\right)^{-s}.\\
\end{align*}
If we separate out the contribution of the difference between the $dth$ coordinates to the innermost sum, and absorb the multiplicative constants, we get
\begin{align*}
&\lesssim q^{-2d-2}\sum_{i_1=1}^{q-1}\sum_{i_1=i_1'+1}^q\sum_{\substack{1\leq i_j,i_j'\leq q\\j=2, \dots, (d-1)}}\sum_{1\leq i_d,i_d' \leq q^2}\\
&\qquad\qquad\left(\left|\frac{1}{\frac{q^2+i_d}{2q^2}}-\frac{1}{\frac{q^2+i_d'}{2q^2}}\right|+\sum_{j=1}^{d-1} \left|\frac{\frac{q+i_j}{2q}}{\frac{q^2+i_d}{2q^2}}-\frac{\frac{q+i_j'}{2q}}{\frac{q^2+i_d'}{2q^2}}\right|\right)^{-s}.\\
\end{align*}
Recalling that each of the differences in the first $(d-1)$ coordinates are dominated by the difference in the first coordinate, as per our reordering, we get that this is bounded above by
\begin{align*}
&\lesssim q^{-2d-2}\sum_{i_1=1}^{q-1}\sum_{i_1=i_1'+1}^q\sum_{\substack{1\leq i_j,i_j'\leq q\\j=2, \dots, (d-1)}}\sum_{1\leq i_d,i_d' \leq q^2}\\
&\qquad\qquad\left(\left|\frac{2q^2}{q^2+i_d}-\frac{2q^2}{q^2+i_d'}\right|+(d-1) \left|\frac{q^2+qi_1}{q^2+i_d}-\frac{q^2+qi_1'}{q^2+i_d'}\right|\right)^{-s}.\\
\end{align*}
Since the summands now no longer depend on the middle $(d-2)$ pairs of indices, we will remove them from consideration. Because the sum was reordered so that $|i_j-i_j'|\leq(i_1-i_1')$ for all $j\in [2.. (d-1)],$ we have control on the number of pairs of indices in each such dimension. More precisely, we have that for any $j\in[2..(d-1)],$ we have that the number of pairs $(i_j,i_j')$ contributing to this sum is no more than $2q(i_1-i_1')$, because we have $q$ choices for $i_j$, then no more than $2(i_1-i_1')$ choices for $i_j'$ within range. So the sum above can be bounded by
\begin{align*}
&\lesssim q^{-2d-2}\sum_{i_1=1}^{q-1}\sum_{i_1=i_1'+1}^q \left(q(i_1-i_1')\right)^{d-2}\sum_{1\leq i_d,i_d' \leq q^2}\\
&\qquad\qquad\left(\left|\frac{2q^2}{q^2+i_d}-\frac{2q^2}{q^2+i_d'}\right|+(d-1) \left|\frac{q^2+qi_1}{q^2+i_d}-\frac{q^2+qi_1'}{q^2+i_d'}\right|\right)^{-s}\\
&\lesssim q^{-d-4}\sum_{i_1=1}^{q-1}\sum_{i_1=i_1'+1}^q(i_1-i_1')^{d-2}\sum_{1\leq i_d,i_d' \leq q^2}\left(\left|\frac{q^2(i'_d-i_d)}{q^4+q^2(i_d+i_d')+i_di_d'}\right|+\right.\\
&\qquad\qquad\left. (d-1) \left|\frac{q^3(i_1-i_1')+q^2(i_d'-i_d)+q(i_1i_d'-i_1'i_d)}{q^4+q^2(i_d+i_d')+i_di_d'}\right|\right)^{-s},\\
\end{align*}
which, by recalling that $1\leq i_d,i_d'\leq q^2$ is bounded above by
\begin{align*}
&\lesssim q^{-d-4}\sum_{i_1=1}^{q-1}\sum_{i_1=i_1'+1}^q(i_1-i_1')^{d-2}\sum_{1\leq i_d,i_d' \leq q^2}\left(\left|\frac{q^2(i'_d-i_d)}{4q^4}\right|+\right.\\
&\qquad\qquad\left. (d-1) \left|\frac{q^3(i_1-i_1')+q^2(i_d'-i_d)+q(i_1i_d'-i_1'i_d)}{4q^4}\right|\right)^{-s}\\
&\lesssim q^{-d-4}\sum_{i_1=1}^{q-1}\sum_{i_1=i_1'+1}^q(i_1-i_1')^{d-2}\sum_{1\leq i_d,i_d' \leq q^2}q^{4s}\\
&\qquad\qquad\left(\left|q^2(i'_d-i_d)\right|+ (d-1) \left|q^3(i_1-i_1')+q^2(i_d'-i_d)+q(i_1i_d'-i_1'i_d)\right|\right)^{-s}.\\
\end{align*}
By adding and subtracting $i_1i_d$ within the final set of parentheses, this expression is
\begin{align*}
&= q^{4s-d-4}\sum_{i_1=1}^{q-1}\sum_{i_1=i_1'+1}^q(i_1-i_1')^{d-2}\sum_{1\leq i_d,i_d' \leq q^2}\left(\left|q^2(i'_d-i_d)\right|\right.+\\
&\qquad\qquad\left. (d-1) \left|q^3(i_1-i_1')+q^2(i_d'-i_d)+q(i_1i_d'-i_1i_d+i_1i_d-i_1'i_d)\right|\right)^{-s}\\
&= q^{4s-d-4}\sum_{i_1=1}^{q-1}\sum_{i_1=i_1'+1}^q(i_1-i_1')^{d-2}\sum_{1\leq i_d,i_d' \leq q^2}\left(\left|q^2(i'_d-i_d)\right|\right.+\\
&\qquad\qquad\left. (d-1) \left|q^3(i_1-i_1')+q^2(i_d'-i_d)+qi_1(i_d'-i_d)+qi_d(i_1-i_1')\right|\right)^{-s}\\
&= q^{4s-d-4}\sum_{i_1=1}^{q-1}\sum_{i_1=i_1'+1}^q(i_1-i_1')^{d-2}\sum_{1\leq i_d,i_d' \leq q^2}\left(\left|q^2(i'_d-i_d)\right|\right.+\\
&\qquad\qquad\left. (d-1) \left|(q^3+qi_d)(i_1-i_1')+(q^2+qi_1)(i_d'-i_d)\right|\right)^{-s}\\
&= q^{4s-d-4}\sum_{i_1=1}^{q-1}\sum_{i_1=i_1'+1}^q(i_1-i_1')^{d-2}\sum_{1\leq i_d,i_d' \leq q^2}(g(i_1, i_1', i_d, i_d'))^{-s},\\
\end{align*}
where we define $$g(i_1, i_1', i_d, i_d'):=|q^2(i_d'-i_d)|+ (d-1) |(q^3+qi_d)(i_1-i_1')+(q^2+qi_1)(i_d'-i_d)|.$$
We now need to look at these terms closely. If the following inequality holds for some positive constant $c$,
\begin{equation}\label{goodMiddle}
(q^3+qi_d)(i_1-i_1')+(q^2+qi_1)(i_d'-i_d)\geq cq^3(i_1-i_1'),
\end{equation}
then we can safely ignore the $q^2(i_d'-i_d)$ part of the term, and we get that
$$g(i_1, i_1', i_d, i_d')\gtrsim q^3(i_1-i_1').$$
However, it's possible that the factor $(i_d'-i_d)$ is so negative, that it could make \eqref{goodMiddle} fail to hold. In this case, we then have that for any constant $c>0$,
$$(q^2+qi_1)(i_d-i_d')>(q^3+qi_d-cq^3)(i_1-i_1').$$
Recalling that $i_1<q,$ this tells us that for any $c>0,$
$$2q^2(i_d-i_d')>(q^3+qi_d-cq^3)(i_1-i_1'),$$
which tells us that again
\begin{align*}
g(i_1, i_1', i_d, i_d')&= |q^2(i_d'-i_d)|+ (d-1) |(q^3+qi_d)(i_1-i_1')+(q^2+qi_1)(i_d'-i_d)|\\
&\gtrsim \frac{1}{2}(q^3+qi_d-cq^3)(i_1-i_1')\gtrsim q^3(i_1-i_1')\\
\end{align*}
So we can continue our analysis of $II$ by bounding each of these terms by some positive constant multiple of $q^3(i_1-i_1')$. Continuing, we now have that
\begin{align*}
II&\lesssim q^{4s-d-4}\sum_{i_1=1}^{q-1}\sum_{i_1=i_1'+1}^q(i_1-i_1')^{d-2}\sum_{1\leq i_d,i_d' \leq q^2}(q^3(i_1-i_1'))^{-s}\\
&=q^{4s-d-4}\sum_{i_1=1}^{q-1}\sum_{i_1=i_1'+1}^q(i_1-i_1')^{d-2}q^4q^{-3s}(i_1-i_i')^{-s}\\
&\lesssim q^{s-d}\sum_{i_1=1}^{q-1}\sum_{i_1=i_1'+1}^q(i_1-i_i')^{d-s-2}\\
&=q^{s-d}\sum_{j=1}^{q-1} (q-j)j^{d-s-2},
\end{align*}
where in the last step, we reparameterized the sum with $j:=i_1-i_1'.$ We now estimate this sum by an integral to get that the previous expression is bounded above by
\begin{align*}
&\leq q^{s-d} \int_1^q (q-x)x^{d-s-2}dx\\
&= q^{s-d}\left(\left[(q-x)\left(\frac{x^{d-s-1}}{d-s-1}\right)\right]_1^q-\int_1^q\frac{x^{d-s-1}}{d-s-1}(-1)dx\right)\\
&=q^{s-d}\left(\frac{-(q-1)}{(d-s-1)}+\frac{q^{d-s}-1}{(d-s-1)(d-s)}\right)\lesssim 1,
\end{align*}
where the last step follows by our assumption that $\frac{d}{2}<s<\frac{d+1}{2}.$

To conclude the energy estimate on $F$, we combine the above bounds on $I$ and $II$ to get
\begin{equation}\label{enF}
I_s'(F)\lesssim 1.
\end{equation}

Finally, we compute the energy between the sets $E$ and $F$. We break up the sum coordinatewise into

\begin{align*}
I_s'(E,F)&=\frac{1}{{n \choose 2}}\sum_{\substack{p \in E,\\ p'\in F}}|p-p'|^{-s}\lesssim n^{-2}\sum_{\substack{p\in E,\\ p'\in F}}\sum_{j=1}^d|p_j-p_j'|^{-s}\\
&= n^{-2}\sum_{\substack{i_j,i'_j\in A\\j=1,\dots, (d-1)}}\sum_{b,b'\in B}\left(\left[\sum_{j=1}^{d-1}\left|\frac{q+i_j}{2q}-\frac{\frac{q+i_j'}{2q}}{b'}\right|\right]+\left|\frac{q^2+b}{2q^2}-\frac{1}{b'}\right|\right)^{-s}\\
&\lesssim q^{-2d-2}\sum_{\substack{i_j,i'_j\in A\\j=1,\dots, (d-1)}}\sum_{b,b'\in B}\left|\frac{q^2+b}{2q^2}-\frac{1}{b'}\right|^{-s}\\
&= q^{-2d-2}(2q^2)^s\sum_{b,b'\in B}(b')^s\sum_{\substack{i_j,i'_j\in A\\j=1,\dots, (d-1)}}\left|b'q^2+bb'-2q^2\right|^{-s}\\
&\lesssim q^{2s-2d-2}\sum_{b,b'\in B}\sum_{\substack{i_j,i'_j\in A\\j=1,\dots, (d-1)}}\left|b'q^2+bb'-2q^2\right|^{-s},\\
\end{align*}
Where we have used the fact that any value of $b'$ is in the interval $\left(\frac{1}{2},1\right]$ in the last step. Using this fact again, we see that $\left|b'q^2+bb'-2q^2\right|\geq \frac{3}{2}q^2-1\geq q^2,$ so we can continue
\begin{align*}
I_s'(E,F)&\lesssim q^{2s-2d-2}\sum_{b,b'\in B}\sum_{\substack{i_j,i'_j\in A\\j=1,\dots, (d-1)}}\left|b'q^2+bb'-2q^2\right|^{-s}\\
&\lesssim q^{2s-2d-2}\sum_{b,b'\in B}\sum_{\substack{i_j,i'_j\in A\\j=1,\dots, (d-1)}}|q^2|^{-s}\\
&\lesssim q^{2s-2d-2}\cdot q^4\cdot q^{2(d-1)}\cdot q^{-2s}\lesssim 1.
\end{align*}
This gives us that
\begin{equation}\label{enEF}
I_s'(E,F)\lesssim 1.
\end{equation}

\subsubsection{Continuous construction}
Given an $s<\frac{d+1}{2},$ and an $\epsilon >0,$ fix a scale $q\approx \epsilon^{-\frac{s}{d+1}}$ and decompose $\mathbb R^d$ into a lattice of half-open cubes of side-length $\epsilon$, that is, translates of $[0,q^{-2})^d.$ By construction, we know that $E\cup F\subset [0,2]^d.$ Select any cube with a point from either $E$ or $F$ in it. Proposition \ref{sep} tells us that no square of side-length $q^{-2}\geq \epsilon$ can have more than a constant number of points in it, so we should have selected about $n$ cubes. Call this set of cubes $R.$ Now, we define $\mu$ to be supported on the union of these cubes, and so that $\int d\mu = 1.$ So we will have that the $\mu$ measure of a cube of side-length $\epsilon$ will be $\epsilon^s.$ More precisely,
\begin{equation}\label{muDef}
d\mu(x) = \epsilon^{s-d}\sum_{r\in R}\chi_r(x)dx.
\end{equation}
Now, we compute the energy of $\mu,$ which by \eqref{enE}, \eqref{enF}, and \eqref{enEF}, is
$$I_s(\mu)\approx I_s'(E\cup F)= I_s'(E)+2I_s'(E,F)+I_s'(F)\lesssim1.$$

Finally we compute the dot product version of the Falconer estimate with $\epsilon = q^{-\frac{d+1}{s}}$ and using Proposition \ref{discreteDP}.
$$(\mu\times\mu)\left\{(x,y)\in[0,2]^d:1 \leq x\cdot y \leq 1+\epsilon \right\}\approx q^{2d}\cdot \epsilon^{2s}\approx \epsilon^{\frac{2s}{d+1}},$$
as claimed.

\end{document}